\documentclass[letterpaper]{article}

\usepackage{graphicx}
\usepackage{enumerate}
\usepackage{enumitem}
\usepackage{amsmath}
\usepackage[english]{babel}
\usepackage{amssymb}
\usepackage{epsfig}
\usepackage{ifpdf}
\usepackage{color}
\usepackage{times}  
\usepackage{stmaryrd}

\newtheorem{theorem}{Theorem}
\newtheorem{corollary}[theorem]{Corollary}

\newtheorem{definition}{Definition}

\newtheorem{lemma}[theorem]{Lemma}
\newtheorem{proposition}[theorem]{Proposition}

\newenvironment{proof}{\par\bproof}{\eproof\(\qed\) \par\medskip}

\newcommand{\qed}{\quad\mbox{\rule{7pt}{7pt}}}

\newcommand{\NN}{\mathbb{N}}
\newcommand{\RR}{\mathbb{R}}

\newcommand{\rmu}{\mathrm{\mathbf{u}}}
\newcommand{\rmv}{\mathrm{\mathbf{v}}}
\newcommand{\rmw}{\mathrm{\mathbf{w}}}
\newcommand{\rmx}{\mathrm{\mathbf{x}}}
\newcommand{\rmy}{\mathrm{\mathbf{y}}}
\newcommand{\rmz}{\mathrm{\mathbf{z}}}
\newcommand{\rme}{\mathrm{\mathbf{e}}}

\newcommand{\set}[2]{\{\,#1\,|\,#2\,\}}

\newcommand{\setof}[1]{\{\,#1\,\}}

\newlength{\first}\newlength{\second}

\newcommand{\smidge}{{\kern .05em}}

\newlength{\subjtolt}

\newlength{\saveparindent}
\def\bproof{\begin{rm}\protect\vspace{5pt}\noindent{\bf Proof: }%
\addtolength{\parskip}{4pt}\setlength{\parindent}{0pt}}
\def\eproof{\end{rm}\addtolength{\parskip}{-4pt}%
\setlength{\parindent}{\saveparindent}}

\newcommand{\bprooff}[1]{\begin{rm}\protect\vspace{5pt}%
\noindent{\bf Proof of #1: }\addtolength{\parskip}{4pt}%
\setlength{\parindent}{0pt}}

\newenvironment{prooff}[1]{\par\bprooff{#1}}{\eproof\qed\par}

\begin{document}
\title{Counting perfect matchings of cubic graphs in the geometric dual}

\author{Andrea Jim\'enez\thanks{
  Depto.~Ing.~Matem\'{a}tica, U.~Chile.
  Web: \texttt{www.dim.uchile.cl/$\sim$ajimenez}.
  Gratefully acknowledges the support of
    MECESUP UCH0607, and CONICYT via Basal in Applied Mathematics and
    FONDECYT 1090227.}
  \and 
  Marcos Kiwi\thanks{
  Depto.~Ing.~Matem\'{a}tica \&
  Ctr.~Modelamiento Matem\'atico UMI 2807, U.~Chile.
  Web: \texttt{www.dim.uchile.cl/$\sim$mkiwi}.
  Gratefully acknowledges the support of
    CONICYT via Basali n Applied Mathematics and
    FONDECYT 1090227.}
}

\maketitle
\begin{abstract}
Lov\'asz and Plummer conjectured, in the mid 1970's, that
  every cubic graph $G$ with no cutedge has an exponential 
  in $|V(G)|$ number of perfect matchings.
In this work we show that every cubic planar graph
  $G$ whose geometric dual graph is a stack
  triangulation has at least $3\varphi^{|V(G)|/72}$ 
  distinct perfect matchings, where $\varphi$ is the golden ratio.
Our work builds on a novel approach relating Lov\'asz and Plummer's 
  conjecture and the number of so called groundstates of the widely
  studied Ising model from statistical physics.
\end{abstract}

\section{Introduction}\label{sec:intro}
A graph is said to be \emph{cubic} if each vertex has degree $3$
and \emph{bridgeless} if it contains no cutedges. 
As early as in 1891 Petersen proved
  that \emph{every cubic bridgeless graph has a perfect matching.}
Nowadays, this famous theorem is obtained 
  indirectly using major results such as Hall's Theorem 
  from 1935 and Tutte's 1-factor Theorem from 1947.

In the mid-1970's, Lov\'asz and Plummer asserted that 
  \emph{for every cubic bridgeless graph with $n$ vertices, 
  the number of perfect matchings is exponential in $n$.}
The best result known is a superlinear lower bound by
Esperet, Kardos and Kr\'al'~\cite{EKK10}.

The conjecture remains open despite considerable attempts to prove.
So far, there are three classes of cubic graphs 
  for which the conjecture has been proved.
For bipartite graphs, the assertion 
  was shown by Voorhoeve~\cite{V} who proved:
  \emph{Every cubic bipartite graph with $n$ vertices has at least 
  $6(4/3)^{\frac{n}{2} -3 }$ perfect matchings.}
This result was later extended to $k-$regular bipartite graphs by 
  Schrijver~\cite{S}.
The conjecture was positively solved for the class of planar graphs
  by Chudnovsky and Seymour~\cite{CS} who showed:
  \emph{Every cubic bridgeless planar graph with $n$ vertices has 
  at least $2^{cn}$ perfect matchings, where $c=1/655978752$.}
Oum~\cite{O} recently established the conjecture for the 
  class of claw-free cubic graphs: 
  \emph{Every claw-free cubic bridgeless graph with $n$ vertices has 
  at least $2^{n/12}$ perfect matchings.}

In what follows, we restrict to the class of planar graphs. 
  We suggest to study the conjecture of Lov\'asz and Plummer in the dual
  setting. This relates the conjecture to a phenomenon well known in
  statistical physics, namely to the degeneracy of the Ising model on
  totally frustrated triangulations of the plane.

A planar graph is a \emph{triangulation} 
  if each face is bounded by a cycle of length $3$. 
Note that the dual graph~$G^*$ of a cubic bridgeless 
  planar graph $G$ is a triangulation. 
A set $M$ of edges of a triangulation~$\Delta$ is 
  \emph{intersecting} if $M$ contains exactly one edge of 
  each face of $\Delta$. Clearly, $M$ is an intersecting
  set of~$G^*$ if and only if  $M$ is a perfect matching
  of $G$. Now, with the previous definitions, 
  we can reformulate the conjecture of Lov\'asz and 
  Plummer for the class of planar graphs as follows: \emph{Each planar
  triangulation has an exponential number of intersecting sets of
  edges}.

Next, let us consider the Ising model. 
  Given a triangulation $\Delta=(V,E)$ we associate the 
  \emph{coupling constant} 
  $c(e)=-1$ with each edge $e\in E$. 
For any $W \subseteq V$, a \emph{spin assignment} of $W$
  is any function $\texttt{s}: W \,
  \rightarrow \, \{1,-1\}$ and $1$, $-1$
  are called \emph{spins}.
A \emph{state} of $\Delta$ 
  is any spin assignment of $V$.
The energy of a state $\texttt{s}$ 
  is defined as $-\sum_{e=\{u,v\}\in E}c(e)\texttt{s}(u)\texttt{s}(v)$.
The states of minimum energy are called \emph{groundstates}.
The number of groundstates is usually called the \emph{degeneracy} of
  $\Delta$, denoted $g(\Delta)$, and it is an extensively studied 
  quantity (for regular lattices) in statistical physics. 
Given a state $\texttt{s}$ of $\Delta$ we say that
  edge $\{u,v\}$ is \emph{frustrated} by $\texttt{s}$
  or that $\texttt{s}$ \emph{frustrates} edge $\{u,v\}$
  if $\texttt{s}(u)=\texttt{s}(v)$. 
Clearly, each state frustrates at least one edge of each
  face of $\Delta$.
A state is a groundstate if it 
  frustrates the smallest possible number of edges. 

We say that a state $\texttt{s}$ \emph{is satisfying for} 
  a face $f$ of a planar triangulation~$\Delta$,
  if there is exactly one edge $e=\{u_1,u_2\}$ 
  in the boundary of $f$ that is \emph{frustrated} by $\texttt{s}$.
Moreover, we say that~$\texttt{s}$ is a 
  \emph{satisfying state} 
  of $\Delta$ if $\texttt{s}$ is satisfying for every inner face 
  $f$ of $\Delta$.
Clearly, the set of edges frustrated
  by a satisfying state which is also 
  satisfying for the outer face is an 
  intersecting set.
Hence, the number of satisfying
  states which are also satisfying for the 
  outer face, is at most twice the number of 
  intersecting sets of edges.
The converse also holds: if we
  delete an intersecting set of edges from a planar triangulation, 
  then we get a bipartite graph and its bipartition determines a 
  satisfying spin assignment which is also satisfying
  for the outer face.
Given that any planar triangulation $\Delta$ has an intersecting set, 
  (induced by a perfect matching in its dual), it follows that 
  $\texttt{s}$ is a satisfying
  state  of $\Delta$ which is also satisfying for its 
  outer face if and only if $\texttt{s}$ is a groundstate of~$\Delta$.
Summarizing, the degeneracy is twice 
  the number of intersecting sets. 
Hence, Chudnovsky and Seymour's result can be reformulated as follows:
  \emph{Each planar triangulation has an exponential (in the 
  number of vertices) degeneracy}.
This motivated Jim\'enez, Kiwi and Loebl~\cite{JKL} to 
  consider the problem of lower bounding the degeneracy of 
  triangulations of an $n$-gon, as well as the use of the (transfer matrix)
  method for achieving their goal.
Since the dual of triangulations of $n$-gons are seldom cubic graphs, the 
  results of~\cite{JKL} do not directly relate to Lov\'asz and
  Plummer's conjecture, not even for a subfamily of cubic graphs.
In this article, we further develop the approach proposed 
  in~\cite{JKL} and establish the feasibility of using it to attack
  Lov\'asz and Plummer's conjecture for a non-trivial subclass of 
  cubic graphs.
More precisely, the subclass of cubic bridgeless planar graphs whose 
  geometric dual are stack triangulations (also called 
  $3$-trees~\cite[page~167]{BBS}). 
Specifically, provided $\varphi=(1+\sqrt{5})/2\approx1.6180$ denotes
  the golden ratio, we establish the following:
\begin{theorem}\label{theo:main}
The degeneracy of any stack triangulation 
  $\Delta$ with $|\Delta|$ vertices is at 
  least $6\varphi^{(|\Delta|+3)/36}$. 
\end{theorem}
As a rather direct consequence of the preceding theorem we obtain the 
  following result.
\begin{corollary}\label{coro:main} 
The number of perfect matchings of a cubic 
  graph $G$ whose dual graph is a 
  stack triangulation is at least $3\varphi^{|V(G)|/72}$.
\end{corollary}
Note that the preceding result applies to a subclass of graphs
  for which Chudnovsky and Seymour's~\cite{CS} work already 
  establishes the validity of Lov\'asz and Plummer's conjecture,
  albeit for a smaller rate of exponential growth and arguably
  by more complicated and involved arguments. 
We believe that the main relevance of this work is that it validates
  the feasibility of the alternative approach proposed in~\cite{JKL}
  for approaching Lov\'asz and Plummer's conjecture. 

\subsection{Organization}
The paper is organized as follows. 
We provide some mathematical background in Section~\ref{sec:prelim}.
Then, in Section~\ref{section:bijection}, 
  we describe a bijection between rooted stack triangulations
  and colored rooted ternary trees --- this bijection
  allows us to work with ternary trees instead of triangulations.
In Section~\ref{section:transfer}, we first introduce the concept
  of degeneracy vector in stack triangulations.
This vector satisfies that the sum of its coordinates is the number of 
  satisfying states of the stack triangulation.
We also introduce the concept of root vector of a ternary trees and show 
  that via the aforementioned bijection, the degeneracy vector of a 
  stack triangulation $\Delta$ is the same as the root vector of 
  the associated colored rooted ternary tree.
In Section~\ref{section:main}, we adapt to our setting the 
  transfer matrix method as used in statistical physics in the 
  study of the Ising Model.
Some essential results are also established.
In Section~\ref{sec:results}, we prove the main results
  of this work.
In Section~\ref{sec:final-comments}, we conclude with a 
  brief discussion and comments about possible future research 
  directions.

\section{Preliminaries}\label{sec:prelim}
We now introduce the main concepts and notation 
  used throughout this work.

\subsection{Stack triangulations}
Let $\Delta_0$ be a triangle. 
For $i \geq 1$, let $\Delta_{i}$ be the plane triangulation obtained 
  by applying the following \emph{growing rule} to $\Delta_{i-1}$.
\begin{quote}
\textbf{growing rule:} Given a plane triangulation $\Delta$,
\begin{enumerate}
\item 
  \label{item:step1}
  Choose an inner face $f$ from $\Delta$,
\item 
  Insert a new vertex $u$ at the interior of $f$.
\item 
  Connect the new vertex $u$ to each vertex of the 
  boundary of $f$.
\end{enumerate}
\end{quote}
Clearly, the number of vertices of $\Delta_{n}$ is $n+3$.
The collection of $\Delta_{n}$'s thus obtained are called 
  \emph{stack triangulation}. 
Among others, the set 
  of stack triangulations coincides with
  the set of plane triangulations having a unique 
  Schnyder Wood (see~\cite{FZ})
  and is the same as the collection of planar $3$-trees 
  (see~\cite[page~167]{BBS}).

Consider now a stack triangulation $\Delta_1$ 
  and for $i \geq 2$, let $\Delta_{i}$ 
  be the plane triangulation obtained by applying the growing rule 
  to $\Delta_{i-1}$ restricting Step~\ref{item:step1} so the face chosen
  is one of the three new faces obtained by the application
  of the growing rule to $\Delta_{i-2}$. 
For $n \geq 1$, we say that $\Delta_{n}$ is a 
  \emph{stack-strip triangulation} 
  (for an example see Figure~\ref{fig:top}).
Clearly, stack-strip triangulations are a subclass of stack triangulations. 

\begin{figure}[h]
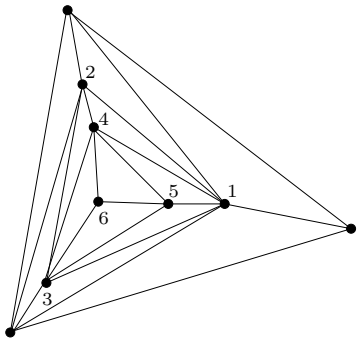

\centering
\ifpdf\input{stack_strip.pdf_t}\else\input{stack_strip.pstex_t}\fi
\caption{Example of stack-strip triangulation (numbers correspond
  to the order in which nodes are added by the growing rule).}\label{fig:top}
\end{figure}

Let $\Delta_n$ be a stack triangulation with $n\geq0$ and $\Delta_0$
  be the starting plane triangle in its construction.  
If we prescribe the counterclockwise orientation to any edge of 
  $\Delta_0$, we say that $\Delta_n$ is a rooted stack triangulation 
  (see Figure~\ref{rooted_stack}).
\begin{figure}[h]
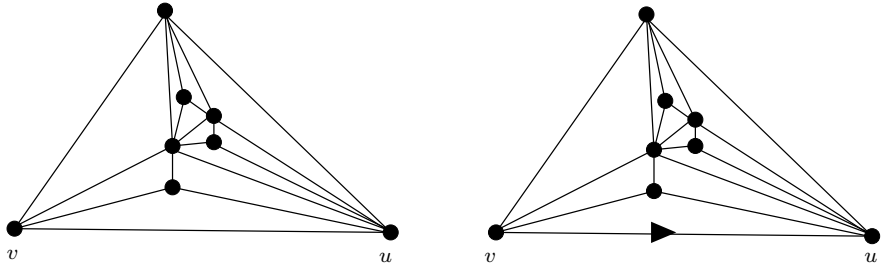

\centering
\ifpdf\input{rooted_stack.pdf_t}\else\input{rooted_stack.pstex_t}\fi
\caption{
A stack triangulation (left) and the rooted stack triangulation 
  obtained by prescribing the counterclockwise orientation to the edge 
  $\{v,u\}$ (right).}\label{rooted_stack}
\end{figure}

\subsection{Ternary trees}
A \emph{rooted tree} is a tree $T$ with a special 
  vertex $v\in V(T)$ designated to be the root. 
If $v$ is the root of~$T$, we denote~$T$ by~$T_v$. 
A \emph{rooted ternary tree} is a rooted tree~$T_v$ such that all its 
  vertices have at most three children. 
{From} now on, let $X$ be an arbitrary set with three elements. 
We say that a rooted ternary tree~$T_v$ is \emph{colored by $X$} 
  (or simply \emph{colored}) if; (1) each non-root vertex is labeled
  by an element of $X$, and~(2) for every vertex of $V(T)$ 
  all its children have different labels.

\section{From stack triangulations to ternary trees}\label{section:bijection}
It is well known that stack triangulations are in bijection with 
  ternary trees (see~\cite{MA}). 
For our purposes,
  the usual bijection is not enough (we need a more 
  precise handle on the way in which triangular faces touch each other). 
The main goal of this section 
  is to precisely describe a one-to-one correspondence better suited 
  for our purposes.

\subsection{Bijection}\label{labprocedure}
Let $\Delta_n$ be a rooted stack triangulation with $n\geq 1$ and 
  $\Delta_0$ be the starting plane triangle in its construction. 
We will show how to construct a 
  colored rooted ternary tree $T(\Delta_n)$
  which will be in bijective correspondence with $\Delta_n$.  

Throughout this section, the following concept will be useful.
\begin{definition} 
Let $\Delta$ be a rooted stack triangulation. 
Let $\tilde{\Delta}$ be the rooted stack triangulation obtained by 
  prescribing the counterclockwise orientation to exactly one edge of 
  each inner face of $\Delta$. 
We refer to $\tilde{\Delta}$ as an auxiliary stack triangulation of $\Delta$.
\end{definition}
Note that in an auxiliary stack triangulation of $\Delta$, we allow inner
  faces of $\Delta$ to have edges oriented clockwise as long
  as exactly one of its edges is oriented counterclockwise.
It is also allowed to have edges with both orientations.

We now, describe the key procedure in the construction of $T(\Delta_n)$.
For $i \in \{1,\ldots,n\}$, let $f_{i}$, $u_{i}$ and
  $\Delta_i$, denote the chosen face, the new vertex and the output
  corresponding to the $i$-th application of the growing rule in the
  construction of $\Delta_n$.
The procedure recursively constructs an auxiliary
  stack triangulations~$\tilde{\Delta}_i$ of~$\Delta_i$.
Initially, $i=1$ and $\tilde{\Delta}_0$ is $\Delta_0$ with one 
  of its edges oriented counterclockwise. 
\begin{quote}
\textbf{Labeling procedure:} 

\textbf{Step 1:}
Let $\vec{e}_{f_i}$ be the counterclockwise oriented edge of $f_i$.  
The orientation of $\vec{e}_{f_i}$ induces a counterclockwise ordering of 
  the three new faces around $u_i$ starting by the face that 
  contains~$\vec{e}_{f_i}$, say $f_i(1)$. 
Let $f_i(2)$ and $f_i(3)$ denote the second and third new faces
  according to the induced order. 
For each $j \in \{1,2,3\}$, we say that $f(j)$ is in position $j$ or that 
  $j$ is the position of $f_i(j)$. (See Figure \ref{rule:bij}.)

\textbf{Step 2}: For each $j \in \{2,3\}$, take the unique edge
  $e_{f_i}(j)$ in $E(f_i) \cap E(f_i({j}))$ and prescribe the counterclockwise
  orientation to this edge (see~Figure~\ref{rule:bij}). 
For all other faces of $\Delta_{i}$ not contained in $f_i$, 
  keep the same counterclockwise oriented edge. 
(Observe that for each $j \in \{1,2,3\}$, the triangle $f_i({j})$ has a 
  prescribed counterclockwise orientation in one of its three edges. 
Moreover, note that $\vec{e}_{f_i}=\vec{e}_{f_i}({1})$.) 
\end{quote}
\begin{figure}[h]
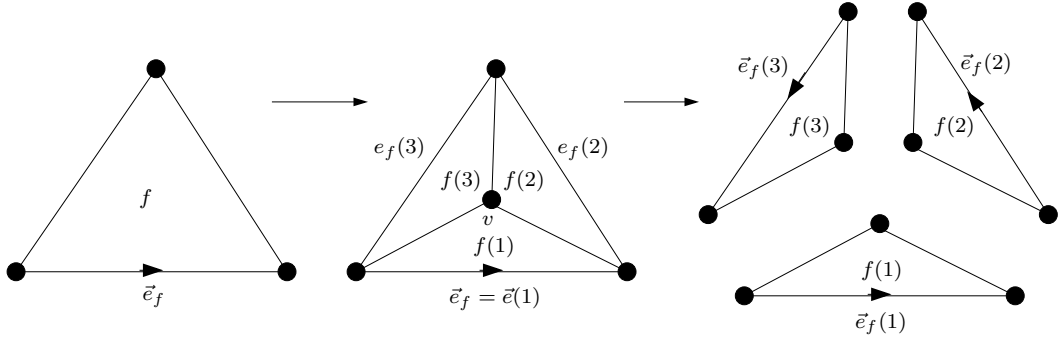

\centering
\ifpdf\input{rule1_bij.pdf_t}\else\input{rule1_bij.pstex_t}\fi
\caption{Labeling procedure. Left to center, step~1. Center to right, step~2.}\label{rule:bij}
\end{figure}

The set 
  $\Theta_{{\Delta}_n}
     =\{(f_{i}, u_{i}, f_{i}(1), f_{i}(2), f_{i}(3))\}_{i\in\{1,\ldots,n\}}$ 
  will be henceforth referred to as 
  the \emph{growth history of $\Delta_n$}.
Note that, for $j\in\{1,2,3\}$, each face $f_{1}(j)$  
  together with its oriented edge induce 
  a rooted stack triangulation, henceforth denoted 
  $\Delta_{n}^{j}$, on the vertices of $\Delta_n$
  that lie on the boundary and interior of $f_{1}(j)$.

We are ready to describe $T(\Delta_{n})$ 
  in terms of the growth history of $\Delta_n$:
\begin{quote}
\textbf{Combinatorial description of $T(\Delta_{n})$:} 
Let $X=\{1,2,3\}$. 
Let $V(T(\Delta_n))= \{u_{1},\ldots,u_{n}\}$. 
Let $u_1$ be the root of $T(\Delta)$.  
For $i \in\{2,\ldots,n\}$, $u_i$ is a child of vertex $u_j$ 
  if there is a $k \in\{1,2,3\}$ such that 
  $f_{i}=f_{j}(k)$. 
The label of $u_i$ is $k$.
For an example see Figure~\ref{bijection}.
\end{quote}
\begin{figure}[h]
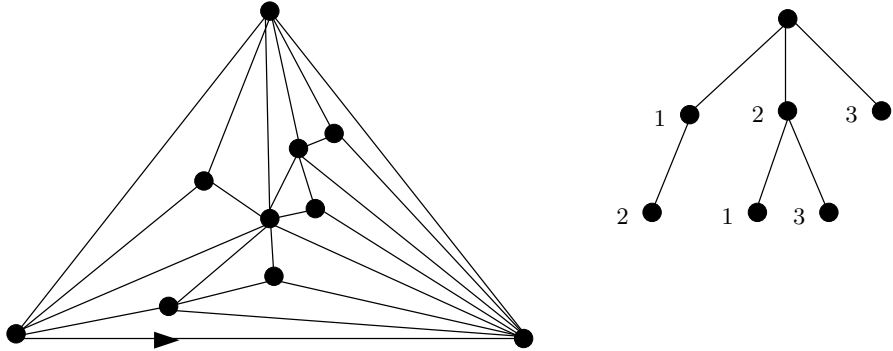

\centering
\ifpdf\input{bijection.pdf_t}\else\input{bijection.pstex_t}\fi
\caption{Example of the bijection between rooted stack triangulations 
  and colored rooted ternary trees.}\label{bijection}
\end{figure}

In particular, we have proved the following result.
\begin{proposition}\label{prop:bij}
Let $\Delta_n$ be a rooted stack triangulation. 
The colored ternary tree $T(\Delta_n)$ rooted on $v$, satisfies 
  the following statements:
\begin{enumerate}
\item \label{item:bij1}
If $\Delta_{n}^{i}$ has 3 vertices for all $i \in \{1,2,3\}$, 
  then $T(\Delta_n)$ has exactly one vertex $v$ (its root).
\item \label{item:bij2}
If there are $i,j \in \{1,2,3\}$ with $i\neq j$ such that 
	$\Delta_{n}^{i}$ and $\Delta_{n}^{j}$ have 3 vertices 
	and $\Delta_{n}^{k}$ with $k \in \{1,2,3\} \setminus \{i,j\}$ 
	has at least 4 vertices, then the root $v$ has exactly one 
	child $w$ labeled by $k$. Moreover, the root of $T(\Delta_{n}^{k})$ 
	is $w$, where $T(\Delta_{n}^{k})$ is the colored 
	sub-ternary tree of $T(\Delta_n)$ induced by $w$ 
	and its descendants.  

\item  \label{item:bij3}
If there is an $i\in \{1,2,3\}$ such that $\Delta_{n}^{i}$ 
	has 3 vertices and $j, k \in \{1,2,3\} \setminus \{i\}$ 
	with $j\neq k$  such that $\Delta_{n}^{j}$ and $\Delta_{n}^{k}$ 
	have at least 4 vertices, then the root $v$ has exactly two 
	children $w_j$ and $w_k$ labeled by $j$ and $k$, respectively.
	Moreover, for every $t \in \{j,k\}$, the root of $T(\Delta_{n}^{t})$ 
	is $w_t$, where $T(\Delta_{n}^{t})$ is the colored 
	sub-ternary tree of $T(\Delta_n)$ induced by $w_t$ 
	and its descendants.  

\item   \label{item:bij4}
If $\Delta_{n}^{i}$, $i \in \{1,2,3\}$,  has at 
	least 4 vertices, then the root $v$ has three children
	$w_1,w_2$ and $w_3$ labeled by $1, 2$ and $3$, respectively.
	Moreover, for every $i\in \{1,2,3\}$, the root of $T(\Delta_{n}^{i})$ 
	is $w_i$, where $T(\Delta_{n}^{i})$ is the colored 
	sub-ternary tree of $T(\Delta_n)$ induced by $w_i$ 
	and its descendants.        
\end{enumerate}
\end{proposition}

\section{Transfer Method} \label{section:transfer}
The main tool we use to carry out our work, is an adaptation of 
  a method (well known among physicist) called the transfer matrix method.
In~\cite{JKL}, we directly apply the transfer matrix method to obtain 
  the number of satisfying states of triangulations of a convex $n$-gon. 
In this work we develop the technique further by 
  considering transfer vectors instead 
  of transfer matrices.

\subsection{Methodology}
In general terms, our aim is to obtain for each stack
  triangulation $\Delta$ a vector $\rmv_{\Delta}$ in $\RR^{4}$ such
  that the sum of its coordinates equals twice the number of satisfying
  states of $\Delta$. 
We now elaborate on this.
Let $n \geq 1$ and $\Delta_{n}$ be a rooted stack triangulation. 
Let $\Delta_0 = (v_1,v_2,v_3)$ denote the starting triangle in the
  construction of $\Delta_{n}$ such that $\{v_1,v_2\}$ is the oriented
  edge with $v_1$ the tail and $v_2$ the head. 
We wish to construct a vector~$\rmv_{\Delta_n} \in \RR^{4}$ such 
  that its coordinates are indexed by the ordered set 
  $I=\{+++,++-,+-+,-++\}$.
For every $\phi \in I$, the $\phi$-th 
  coordinate of $\rmv_{\Delta_n}$, denoted $\Delta_n[\phi]$, is 
  defined as 
  the number of satisfying states of~$\Delta_{n}$ when the spin assignment
  of $(v_1,v_2,v_3)$ is equal to $\phi$.  
The vector $\rmv_{\Delta_n}$ will be called the 
  \emph{degeneracy vector} of $\Delta_n$. 
  In particular, $\rmv_{\Delta_0}=(0,1,1,1)^{t}$ is the degeneracy vector
 of a triangle.
Clearly, for every $\phi\in I$ we have the relation 
\begin{equation}\label{simm} 
\Delta_{n}[\phi]=\Delta_{n}[-\phi]
\end{equation}

Let $\Theta_{{\Delta}_n}
  =\{(f_{i}, u_{i}, f_{i}(1), f_{i}(2), f_{i}(3))\}_{i \in \{1,\ldots,n\}}$
  be the growth history of $\Delta_n$. 
Let $v$ denote $u_{1}$.   
Recall that $f_{1}(j)$ 
  induces a rooted stack triangulation $\Delta_{n}^{j}$ according to 
  the growth history of $\Delta_n$, 
  (see 
  Subsection~\ref{labprocedure}): 
  the oriented edge of 
  $\Delta_{n}^{1}$ is $\{v_1,v_2\}$ with $v_1$ its tail and $v_2$ its head; 
  the oriented edge of $\Delta_{n}^{2}$ is $\{v_2,v_3\}$ with $v_2$ its tail 
  and $v_3$ its head; and the oriented edge of 
  $\Delta^{3}_{n}$ is $\{v_3,v_1\}$ with $v_3$ its tail and 
  $v_1$ its head.

The following result shows how to express the degeneracy vector of $\Delta_{n}$
  in terms of the degeneracy vectors $\rmv_{\Delta_{n}^{1}}$,
  $\rmv_{\Delta_{n}^{2}}$, and $\rmv_{\Delta_{n}^{3}}$.
\begin{proposition} \label{prop:degvect}
For each $j \in \{1,2,3\}$, let 
  $\rmv_{\Delta_{n}^{j}} = (v_{j}^{k})_{k \in \{0,1,2,3\}}$. Then, 
\[
\rmv_{\Delta_{n}}=  \begin{pmatrix}
		v_{1}^{0}v_{2}^{0}v_{3}^{0} + v_{1}^{1}v_{2}^{1}v_{3}^{1}\\
		v_{1}^{0}v_{2}^{2}v_{3}^{3} + v_{1}^{1}v_{2}^{3}v_{3}^{2}\\
		v_{1}^{2}v_{2}^{3}v_{3}^{0} + v_{1}^{3}v_{2}^{2}v_{3}^{1}\\
		v_{1}^{2}v_{2}^{1}v_{3}^{3} + v_{1}^{3}v_{2}^{0}v_{3}^{2}
		\end{pmatrix}\,.
\]
\end{proposition}
\begin{proof} 
Let $\phi\in I$.
Note that $\rmv_{\Delta_n}[\phi]$ equals the sum 
  of the number of satisfying states of~$\Delta_n$ when 
  $(v_1,v_2,v_3,v)$ are assigned spins $(\phi,+)$ and $(\phi,-)$.
For a given spin assignment to $(v_1,v_2,v_3,v)$, the number of 
  satisfying states of~$\Delta_n$, is obtained by 
  multiplying the number of satisfying states of each $\Delta_{n}^{i}$ 
  when the spin assignment of its outer faces agree with the fixed 
  spins assigned to $(v_1,v_2,v_3,v)$.

First, consider the case where $\phi=+++$.
If~$v$'s spin is $+$, then 
\[
  \Delta_{n}^{1}[+++]\cdot\Delta_{n}^{2}[+++]\cdot\Delta_{n}^{3}[+++] 
  = v_{1}^{0}v_{2}^{0}v_{3}^{0}\,.
\]
If $v$'s spin is $-$, then
\[
  \Delta_{n}^{1}[++-]\cdot\Delta_{n}^{2}[++-]\cdot\Delta_{n}^{3}[++-] 
  = v_{1}^{1}v_{2}^{1}v_{3}^{1}\,.
\]
Hence, $\Delta_{n}[\phi]=v_{1}^{0}v_{2}^{0}v_{3}^{0}+v_{1}^{1}v_{2}^{1}v_{3}^{1}$.

Now, consider the case where $\phi=++-$.
If $v$'s spin is $+$, then 
\[
 \Delta_{n}^{1}[+++]\cdot\Delta_{n}^{2}[+-+]\cdot\Delta_{n}^{3}[-++] 
  = v_{1}^{0}v_{2}^{2}v_{3}^{3}\,.
\]
Recalling that by identity~(\ref{simm})
  we have that $\Delta^{2}_{n}[+--]=\Delta^{2}_{n}[-++]$ 
  and $\Delta^{3}_{n}[-+-]=\Delta^{3}_{n}[+-+]$, 
  if $v$'s spin is $-$, then
\[
  \Delta_{n}^{1}[++-]\cdot\Delta_{n}^{2}[+--]\cdot\Delta_{n}^{3}[-+-] 
  = v_{1}^{1}v_{2}^{3}v_{3}^{2}\,.
\]
Hence, $\Delta_{n}[\phi]=v_{1}^{0}v_{2}^{2}v_{3}^{3}+v_{1}^{1}v_{2}^{3}v_{3}^{2}$.

The other two remaining cases, where $\phi$ equals 
  $+-+$ and $-++$, can be similarly dealt with and left to 
  the interested reader.
\end{proof}

\subsection{Root vectors of ternary trees}
We will now introduce the concept of \emph{root vector} 
  of a colored rooted ternary tree. 
Then, we will see that~$\rmv_{\Delta}$ is the degeneracy vector of 
  the rooted stack triangulation $\Delta$ if and only if 
  $\rmv_{\Delta}$ is the root vector of the colored rooted 
  ternary tree $T(\Delta)$.

Let $T$ be a colored rooted ternary tree. 
For any node
  $u$ of $T\setminus \{v\}$, we denote by $l_{u} \in \{1,2,3\}$ its label.
\begin{definition}\label{def:rules}
Let $T$ be a colored ternary tree rooted at $v$. 
We recursively define the root vector 
  $\rmv \in \RR^{4}$ of~$T$ associated to 
  $v$ according to the following rules:
\begin{quote}
\textbf{Rule 0:} 
$\rmv = (1,1,1,1)^{t}$ when $v$ does not have any children.

\textbf{Rule 1:}
If $v$ has exactly one child $u$ with 
  $\rmu = (u_{s})_{s=0,\ldots,3}$,
  then $\rmv\in[\rmu]$ where
\[
[\rmu]=\left\{(u_1,u_0+u_1,u_3,u_2)^{t}\,,
            (u_1,u_3,u_2,u_0+u_1)^{t}\,, (u_1,u_2,u_0+u_1,u_3)^{t}\right\}\,.
\]
The choice of $\rmv$ depends on the label of $u$; 
  if $l_{u}=i$, $\rmv$ is the $i$-th vector in $[\rmu]$.
 
\textbf{Rule 2:} 
If $v$ has two children $u$ and $w$ with $\rmu =
		(u_{s})_{s=0,\ldots,3}$, $\rmw =
		(w_{s})_{s=0,\ldots,3}$, and
 $(l_{u},l_{w}) \in \{(1,2),(2,3),(3,1)\}$, then
\[\rmv \in  
  \left\{\begin{pmatrix}
          u_1 w_1 \\
	  u_0 w_2{+}u_1 w_3 \\
	  u_3 w_2 \\
	  u_2 w_1{+}u_3 w_0
          \end{pmatrix}\,,  
  \begin{pmatrix}
          u_1 w_1 \\
	  u_3 w_2 \\
	  u_3 w_0{+}u_2 w_1 \\
	  u_1 w_3{+}u_0 w_2
          \end{pmatrix}\,, 
  \begin{pmatrix}
          u_1 w_1 \\
	  u_3 w_0{+}u_2 w_1 \\
	  u_0 w_2{+}u_1 w_3 \\
	  u_3 w_2
          \end{pmatrix}\, 
 \right\}\,.
\]
The choice of $\rmv$ depends on $(l_{u},l_{w})$; 
if $l_{u}=i$, $\rmv$ is the $i$-th vector in the last set.

\textbf{Rule 3:}
If $v$ has three children $u$, $w$ and $z$ with 
  $\rmu = (u_{s})_{s=0,\ldots,3}$,
  $\rmw = (w_{s})_{s=0,\ldots,3}$,
  $\rmz = (z_{s})_{s=0,\ldots,3}$, and $(l_{u},l_{w},l_{z})=(1,2,3)$,
then 
\[
\rmv \in  \left\{
  \begin{pmatrix}
          u_0 w_0 z_0{+}u_1 w_1 z_1 \\
	  u_0 w_2 z_3{+}u_1 w_3 z_2 \\
	  u_2 w_3 z_0{+}u_3 w_2 z_1 \\
	  u_2 w_1 z_3{+}u_3 w_0 z_2
          \end{pmatrix} \right\}\,.
\]
\end{quote}
\end{definition}
The following result establishes that determining the 
  degeneracy vector of rooted stack triangulations is
  equivalent to determining the root vector of 
  colored rooted ternary trees.
\begin{lemma} \label{lemma:deg_root}
Let $n \geq 1$ and $\Delta_n$ be the 
  rooted stack triangulation $\Delta_n$. 
Then, the root vector of the colored ternary tree 
  $T(\Delta_n)$ in bijection with $\Delta_n$
  equals the degeneracy vector of $\Delta_n$.
\end{lemma}
\begin{proof} 
By induction on $n$. 
For the base case $n=1$; the stack triangulation $\Delta_1$ 
is isomorphic to $K_4$ and $T(\Delta_1)$ is a vertex. 
It is clear that $\Delta_1[\phi]=1$ for all $\phi \in I$, and 
the root vector of $T(\Delta_1)$ is obtained by Rule~0
in Definition~\ref{def:rules}. 

Now, let $\Delta_n$ be a rooted stack 
  triangulation with $n>1$. 
We denote by $v$ the root of $T(\Delta_n)$.
We separate the proof in cases according to 
  the number of vertices of the
  rooted stack triangulations $\Delta_{n_i}=
  \Delta_{n}^{i}$ with $i \in \{1,2,3\}$.
We note that if $n_i=0$ for every $i \in \{1,2,3\}$, 
  then $n=1$. 
Thus, we can assume that $n_i\geq1$ for at least one 
  index $i \in \{1,2,3\}$.
We now consider three possible situations.

First, assume there are $i,j \in \{1,2,3\}$ with 
  $i\neq j$ and $k \in \{1,2,3\}\setminus \{i,j\}$ such that  
  $n_i=n_j=0$ and $n_{k}\geq 1$. 
By definition of the degeneracy vector, we have that 
    $\rmv_{\Delta_{n_i}}=\rmv_{\Delta_{n_j}} = (0,1,1,1)^{t}$.
Let $\rmv_{\Delta_{n_k}}= (v_{k}^{t})_{t \in \{0,1,2,3\}}$. 
According to Proposition~\ref{prop:degvect}, we have that
\[
\rmv_{\Delta_{n}} \in  \left\{ \begin{pmatrix}
		v_{1}^{1}\\
		v_{1}^{0}+ v_{1}^{1}\\
		v_{1}^{3}\\
		v_{1}^{2}
		\end{pmatrix},
\begin{pmatrix}
		v_{2}^{1}\\
		v_{2}^{3}\\
		v_{2}^{2}\\
		v_{2}^{1}+ v_{2}^{0}
		\end{pmatrix},
\begin{pmatrix}
		v_{3}^{1}\\
		v_{3}^{2}\\
		v_{3}^{0} + v_{3}^{1}\\
		v_{3}^{3} 
		\end{pmatrix}\,\right\}\,,
\]   
where $\rmv_{\Delta_{n}}$ is 
  the $k$-th vector in the set above. 
Statement~\ref{item:bij2}  of Proposition~\ref{prop:bij} says   
  that $T(\Delta_{n_k})$ is labeled by $k$ and rooted 
  on $w$, where $w$ is the unique child of $v$.
Given that $1 \leq n_k < n$, 
  by induction we get that $\rmw
  =\rmv_{\Delta_{n_k}}$.
By Definition~\ref{def:rules}, we know that 
  $\rmv$ is obtained
  from $\rmw$ by application of Rule~1. 
Hence, $\rmv=\rmv_{\Delta_{n}}$. 

Assume now that there is an $i\in \{1,2,3\}$ such that 
  $n_{i}=0$ and $j, k \in \{1,2,3\} \setminus \{i\}$ 
  with $j\neq k$  such that $n_j, n_k \geq 1$. 
We have that $\rmv_{\Delta_{n_i}} = (0,1,1,1)^{t}$. 
Consider $\rmv_{\Delta_{n_j}} = (v_{j}^{t})_{t \in \{0,1,2,3\}}$ 
  and $\rmv_{\Delta_{n_k}} = (v_{k}^{t})_{t \in \{0,1,2,3\}}$. 
Proposition~\ref{prop:degvect} implies that
 \[
\rmv_{\Delta_{n}} \in  \left\{ \begin{pmatrix}
		v_{2}^{1}v_{3}^{1}\\
		v_{2}^{3}v_{3}^{2}\\
		v_{2}^{3}v_{3}^{0} + v_{2}^{2}v_{3}^{1}\\
		v_{2}^{1}v_{3}^{3} + v_{2}^{0}v_{3}^{2}
		\end{pmatrix},
\begin{pmatrix}
		v_{1}^{1}v_{3}^{1}\\
		v_{1}^{0}v_{3}^{3} + v_{1}^{1}v_{3}^{2}\\
		v_{1}^{2}v_{3}^{0} + v_{1}^{3}v_{3}^{1}\\
		v_{1}^{2}v_{3}^{3} 
		\end{pmatrix},
\begin{pmatrix}
		v_{1}^{1}v_{2}^{1}\\
		v_{1}^{0}v_{2}^{2}+ v_{1}^{1}v_{2}^{3}\\
		v_{1}^{3}v_{2}^{2}\\
		v_{1}^{2}v_{2}^{1}+ v_{1}^{3}v_{2}^{0}
		\end{pmatrix}
\right\}\,,
\]
where $\rmv_{\Delta_{n}}$ is the $i$-th vector in the set above.
Statement~\ref{item:bij3} of Proposition~\ref{prop:bij} guarantees
  that the root $v$ of $T(\Delta)$ 
  has exactly two children $w$ and $u$ 
  labeled $j$ and $k$, respectively.
Moreover, $T(\Delta_{n}^{j})$ and $T(\Delta_{n}^{k})$ are rooted on $w$
  and $u$, respectively. 
We know that $1 \leq n_j < n$ and $1 \leq n_k < n$,
  then by induction, $\rmw=\rmv_{\Delta_{n_j}}$
  and $\rmu=\rmv_{\Delta_{n_k}}$. 
If we now apply Rule~2 of Definition~\ref{def:rules}, we get
  $\rmv=\rmv_{\Delta_{n}}$.

Finally, assume that $n> n_j \geq 1$ for every $j \in \{1,2,3\}$.
Suppose that $\rmv_{\Delta_{n_j}} = (v_{j}^{t})_{t \in \{0,1,2,3\}}$ 
  for each $j \in \{1,2,3\}$.
Statement~\ref{item:bij4} of Proposition~\ref{prop:bij}
  and the induction hypothesis imply that the root $v$ of $T(\Delta)$
  has three children $w_1$, $w_2$ and $w_3$ such that 
  $\rmw_j=\rmv_{\Delta_{n_j}}$ for each $j \in \{1,2,3\}$.
By Proposition~\ref{prop:degvect} and since $\rmv$ is derived by 
   applying Rule~3 of Definition~\ref{def:rules},
  the desired conclusion follows.
\end{proof}

\section{Colored Rooted Ternary Trees}\label{section:main}
The goal of this section is to prove a result that we should
  refer to as the Main Lemma which shows that the degeneracy
  of stack triangulations is exponential in the number of its
  nodes.  

We now introduce notation that will be useful when dealing with
  rooted ternary trees.
We denote by $|T|$ the number of vertices of the ternary tree $T$. 
For any node $u$ of $T$, we denote by $T_{u}$ the colored rooted
  sub-ternary tree of $T$ rooted at $u$ and induced by $u$
  and its descendants.  
Also, 
  we denote by~$P_{\tilde{w},w}$ any path with end nodes $\tilde{w}$ and $w$.
Moreover, $||P_{\tilde{w},w}|| = |P_{\tilde{w},w}|-1$
denotes the length of $P_{\tilde{w},w}$.

\subsection{Remainders}\label{subsec:remain}
In this subsection we introduce the concept of $\emph{remainder}$ 
  of a rooted ternary tree and prove some useful and fundamental claims
  related to this concept.
We will show that after removing remainders from a rooted ternary tree
  we are still left with a tree of size at least a third of the 
  original one. 
The root vertex of the derived remainder free tree will provide a 
  component wise lower bound on the components of the root vertex 
  of the original rooted ternary tree. 
The underlying motivation for this section is that lower bounding
  the components of a root vertex is significantly easier for
  remainder free rooted ternary trees.

\begin{definition}\label{def:remain} 
Let $v$ be a leaf of $T$ and $w$ be its father. 
Consider the following cases:
\begin{enumerate}[label=\Roman{*}.-,ref=\Roman{*}]
\item \label{3ch} If $u\neq v$ is a child of $w$,
  then $|T_{u}| \geq 3$.
\item \label{1ch} If $T_{w}$ is just the edge $wv$,
  then the father of $w$, say $y$, has two children $w$ and $u$, 
    where $|T_{u}|\geq3$.
\end{enumerate}
If Case~\ref{3ch} holds, we say that $\{v\}$ is a
  \emph{remainder} of $T$ and that $w$ is the \emph{generator} 
  of $\{v\}$. 
If Case~\ref{1ch} holds we say that $\{v,w\}$ is a 
  \emph{remainder} of $T$ and that $y$ is its \emph{generator}. 
We say that $T$ is reminder free if it does not contain any remainder. 
We denote the set of remainders of $T$ by $R(T)$ and by $G(R(T))$ 
  the set of its generators.
\end{definition}
See Figures~\ref{remainder_1} and~\ref{remainder_2} 
  for an illustration of the distinct situations encompassed by each
  of the preceding definition's cases.

\begin{figure}[h]
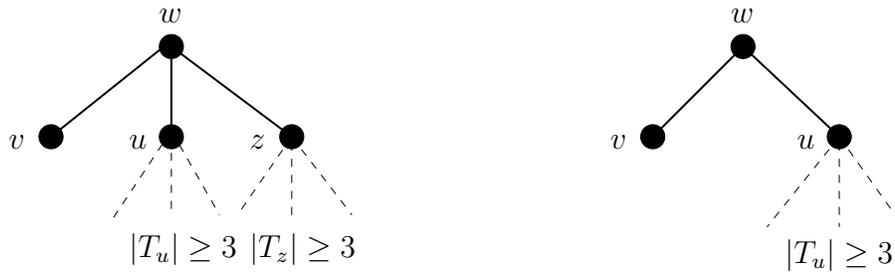

\centering
\ifpdf\input{remainders.pdf_t}\else\input{remainders.pstex_t}\fi
\caption{Structure of $T_{w}\subseteq T$ having a remainder $v$ of $T$
  with generator $w$. 
  Case where $w$ has three children (left) and two children 
  (right).}\label{remainder_1}
\end{figure}
\begin{figure}[h]
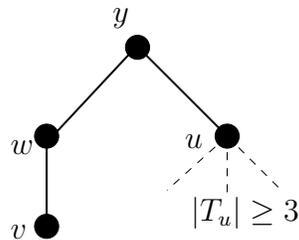

\centering
\ifpdf\input{remainders_1ch.pdf_t}\else\input{remainders_1ch.pstex_t}\fi
\caption{Structure of $T_{y}\subseteq T$ having a remainder 
  $\{v,w\}$ with generator $y$.}\label{remainder_2}
\end{figure}

\begin{proposition}\label{bound}
Let $T$ be a rooted ternary tree.
Then,  $|R(T)|= |G(R(T))|$. 
\end{proposition}
\begin{proof} 
It is enough to show that any vertex $w \in G(R(T))$ is the generator
  of exactly one remainder of~$T$. 
For the sake of contradiction, suppose that $w$ is the generator of 
  at least two remainders of $T$, say $S_1$ and $S_2$.
We consider three possible cases which cover all possible 
  scenarios:
  (i) $S_1=\{v\}$ and $S_2=\{u\}$, 
  (ii) $S_1=\{v,\tilde{v}\}$ and $S_2=\{u,\tilde{u}\}$, and
  (iii) $S_1=\{v\}, S_2=\{u,\tilde{u}\}$.

If $S_1=\{v\}$ and $S_2=\{u\}$, then by Case~\ref{3ch} 
  of Definition~\ref{def:remain}, we get that $|T_{v}|\geq 3$.
If $S_1=\{v,\tilde{v}\}$ and $S_2=\{u,\tilde{u}\}$, then
  by Case~\ref{1ch} of Definition~\ref{def:remain}, we get that
  $|T_{\tilde{v}}|\geq 3$.
If $S_1=\{v\}$ and $S_2=\{u,\tilde{u}\}$, then 
  by Case~\ref{1ch} of Definition~\ref{def:remain}, 
  we have that $|T_{v}|\geq 3$.
Hence, all feasible cases lead to contradictions.
\end{proof}

Let $V_{R(T)}$ denote the subset of vertices of $T$ which belong to the 
  elements of $R(T)$,
  i.e.~$V_{R(T)}~=~\cup_{S\in R(T)}\{v: v \in S\}$.

\begin{lemma}\label{inters}
Let $T$ be a rooted ternary tree.
Then, $\tilde{T} = T \setminus V_{R(T)}$ is remainder free.
\end{lemma}
\begin{proof}
For the sake of contradiction, assume $S$ is a remainder of $\tilde{T}$. 
We consider three scenarios depending on which case of 
  Definition~\ref{def:remain} holds for $S$.

First, assume $S=\{v\}$ satisfies Case~\ref{3ch} of Definition~\ref{def:remain}
  and the father of $v$ in $\tilde{T}$ has tree
  children $v,u,z$ with $|\tilde{T}_{u}|, |\tilde{T}_{z}|\geq3$. 
Clearly, $|T_{u}|\geq|\tilde{T}_{u}|$ and $|T_{z}|\geq|\tilde{T}_{z}|$. 
Since $v$ has no children in $\tilde{T}$, 
  it follows that $v \notin G(R(T))$.
Thus, $v$ is a leaf of $T$ and $\{v\} \in R(T)$.

Assume now that $S=\{v\}$ satisfies Case~\ref{3ch} of 
  Definition~\ref{def:remain} and $v$'s father in $\tilde{T}$, say $w$, 
  has two children $v,u$ with $|\tilde{T}_{u}|\geq 3$. 
We have $|T_{u}|\geq|\tilde{T}_{u}|\geq 3$.
Moreover, since $v$ is a leaf of $\tilde{T}$, it must also hold 
  that~$v$ is a leaf of $T$ (otherwise, all of $v$'s children in $T$
  must belong to some reminder, a situation that is not possible). 
If $w$ has three children in $T$, say $v,u,z$, then $\{z\}\in R(T)$.
This implies that $|T_{v}|\geq 3$, contradicting the fact that
  $v$ is a leaf of $T$. 
Hence, $w$ has two children in $T$.
It follows that $\{v\} \in R(T)$. 

Finally, assume $S=\{v,\tilde{v}\}$ satisfies Case~\ref{1ch} of 
  Definition~\ref{def:remain}. 
Let $w$ be the generator of $S$ and the father of $\tilde{v}$ in $\tilde{T}$. 
Then, $w$ has two children $\tilde{v},u$ in $\tilde{T}$
  with $|\tilde{T}_{u}| \geq 3$. 
We again have that $|T_{u}|\geq|\tilde{T}_{u}| \geq 3$ and that 
  $v$ is a leaf of $T$.
Assume $w$ has three children in $T$, say $\tilde{v},u,z$.
Then, $\{z\}\in R(T)$, implying that $|T_{\tilde{v}}|\geq 3$, 
  and hence $\tilde{v} \in G(R(T))$. 
Therefore, $|T_{v}|\geq 3$, but this cannot happen because $v$ is a 
  leaf of $T$.
Thus, $w$ must have only two children in $T$.
If $\tilde{v}$ has exactly two children in $T$,
  then $\tilde{v} \in G(R(T))$ and $|T_{v}|\geq 3$, contradicting again
  the fact that $v$ is a leaf.
If  $\tilde{v}$ has only one child, then $\{v,\tilde{v}\}\in R(T)$,
  which contradicts the fact that $\tilde{v}$ is a node of $\tilde{T}$.

Since all possible scenarios lead to a contradiction, the 
  desired conclusion follows.
\end{proof}

\begin{lemma}\label{lemma:b} 
Let $\tilde{T} = T \setminus V_{R(T)}$.
Then, $|\tilde{T}|\geq |T|/3$. 
\end{lemma}
\begin{proof}
Follows from the fact that $G(R(T))$ and $R(T)$ are disjoint,
  that each element $S\in R(T)$ is of cardinality at most $2$,
  and Proposition~\ref{bound}.
\end{proof}

\subsection{Counting satisfying states}\label{subsec:css}
In this section, we establish properties of the root 
  vectors of colored rooted ternary trees  
  and relate them to characteristics of colored tree.
Informally, for some special classes of
  colored rooted ternary trees, 
  we obtain lower bounds for the sum of the coordinates of its
  associated rooted vectors. 

Recall that $\varphi=(1+\sqrt{5})/2\approx 1.6180$ denotes the golden ratio.
For $s\in\setof{0,\ldots,3}$, let $e_{s}\in\NN$ and
  $\rme = (\varphi^{e_s})_{s=0,\ldots,3}$.
Define
\begin{eqnarray*}
  \Psi(\rme)  =  2\sum_{j=1}^{3}e_{j}\,, & \text{and} &
  \Phi(\rme)  =  \Psi(\rme) -\left|\set{s}{e_{s}> e_{0}}\right|\,.
\end{eqnarray*}

Henceforth, for a vector $\rmv$ we let 
  $\llbracket v\rrbracket$ denote the collection of all 
  vectors obtained by fixing the first coordinate of 
  $\rmv$ and permuting the remaining coordinates in
  an arbitrary way.
Note that if 
  $\rme = (\varphi^{e_s})_{s=0,\ldots,3}$ with 
  $e_0,e_1,e_2,e_3\in\NN$, then for all 
  $\tilde{\rme}\in\llbracket\rme\rrbracket$ we have that
  $\Psi(\tilde{\rme})=\Psi(\rme)$ and 
  $\Phi(\tilde{\rme})=\Phi(\rme)$.
For a set $S$ of vectors, we let $\llbracket S\rrbracket$
  denote the union of the sets $\llbracket \rmv\rrbracket$
  where $\rmv$ varies over $S$.

Given vectors $\rmx=(x_s)_{s=0,\ldots,3}$ and 
  $\rmy=(y_s)_{s=0,\ldots,3}$, we write 
  $\rmx \geq \rmy$ if $x_s \geq y_s$ for all $s\in \{0,\ldots,3\}$.
\begin{proposition}\label{prop:or2} 
Let $T_v$ be a colored rooted ternary tree with $|T_v|=2$. 
Then, there are $e_0,e_1,e_2,e_3\in\NN$ such that 
  $\rmv \geq \rme=(\varphi^{e_s})_{s=0,\ldots,3}$ 
  and $\Psi(\rme) = 2$.
\end{proposition}
\begin{proof} 
Clearly $T_v$ is a rooted tree on $v$ with exactly one child $w$ which
  is a leaf of $T_v$. 
In other words, $T_v = P_{w,v}$ with $||P_{w,v}||=1$.
We observe that by applying Rules~0 and~1, we get that
  $\rmw = (1,1,1,1)^{t}$ and 
  $\rmv~\in~\llbracket (1,2,1,1)^{t}\rrbracket$.
Given that $1 = \varphi^{0}$ and $2 \geq \varphi^{1}$, it is easy to 
  see that the desired vector $\rme$ belongs to 
  $\llbracket (\varphi^{0}, \varphi^{1}, \varphi^{0}, \varphi^{0})\rrbracket$
\end{proof}

\begin{proposition}\label{prop:or3} 
Let $T_v$ be a colored rooted ternary tree with $|T_v|=3$. 
Then, there are $e_0,e_1,e_2,e_3\in\NN$ such that $\rmv \geq \rme
	=(\varphi^{e_s})_{s=0,\ldots,3}$ and $\Psi(\rme) = 4$.
\end{proposition}
\begin{proof} 
Since $|T_v|=3$, either $T_v = P_{w,v}$ with $||P_{w,v}||=2$, or
  $v$ has exactly two children $w$ and $u$, which are leaves of $T_v$.

In the first scenario, applying Rule~0 once and Rule~1 twice, we get that 
  $\rmv \in \llbracket (2,3,1,1)^{t},(1,2,2,1)^{t}\rrbracket$.
Given that $1 = \varphi^{0}$, $2 \geq \varphi^{1}$ and $3 \geq \varphi^{2}$, 
  we can take $\rme\in 
  \llbracket (\varphi^{1},\varphi^{2},\varphi^{0},\varphi^{0})^{t},
       (\varphi^{0},\varphi^{1},\varphi^{1},\varphi^{0})^{t}\rrbracket$
  satisfying the statement.

In the second scenario, applying Rule~0, we get that 
  $\rmw$ and $\rmu$ are vectors all of whose coordinates are $1$. 
Applying Rule~2, we see that
  $\rmv \in \llbracket (1,2,1,2)^{t}\rrbracket$.
Given that $1 = \varphi^{0}$ and $2 \geq \varphi^{1}$, the desired 
  vector $\rme$ may be chosen
  from the set
  $\llbracket (\varphi^{0},\varphi^{1},\varphi^{0},\varphi^{1})^{t}\rrbracket$.
\end{proof}

\begin{proposition}\label{prop:or4} 
Let $T_v$ be a colored rooted ternary tree with $|T_v|=4$. 
Then, there are $e_0,e_1,e_2,e_3\in\NN$ such that $\rmv \geq \rme
  =(\varphi^{e_s})_{s=0,\ldots,3}$ and $\Psi(\rme) \geq 6$.
\end{proposition}
\begin{proof}
The tree $T_v$ may be one of the four trees depicted in Figure~\ref{or4cases}.
Each case is analyzed separately below (in the order in which they
  appear in Figure~\ref{or4cases}).
\begin{figure}[h]
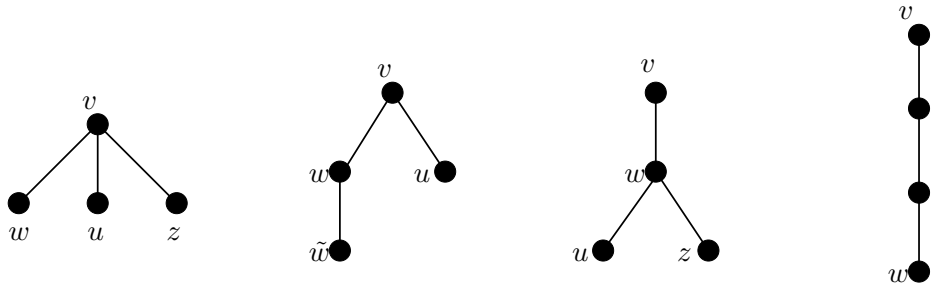

\centering
\ifpdf\input{or4cases.pdf_t}\else\input{or4cases.pstex_t}\fi
\caption{All rooted ternary trees with 4 vertices.}\label{or4cases}
\end{figure}

For the first case, note that by Rule~0 we have that
  $\rmw$, $\rmu$ and $\rmz$ are vectors all of 
  whose coordinates are $1$.
Thus, by Rule~3, we get that 
  $\rmv = (2,2,2,2)^{t} $.
Hence, $\rmv\geq \rme$ where 
  $\rme= (\varphi^{1}, \varphi^{1}, \varphi^{1}, \varphi^{1})^{t}$.

For the second case, by Rule~0 we have that all coordinates of 
  $\rmu$ and $\tilde{\rmw}$ are $1$.
Thus, by Rule~1, $\rmw~\in~\llbracket (1,2,1,1)^{t}\rrbracket$.
 Then, by Rule~2, we get that 
  $\rmv
    \in\llbracket (2,3,1,2)^{t}, (1,2,1,3)^{t}, (1,3,2,2)^{t}\rrbracket$.
Given that $1 = \varphi^{0}$, $2 \geq \varphi^{1}$ and 
  $3 \geq \varphi^{2}$ the result follows.

For the third case, note that $|T_w|=3$ and that the structure of $T_w$ is 
  the same as the second one considered in the proof of 
  Proposition~\ref{prop:or3}. 
Hence, we know that $\rmw \in \llbracket (1,2,1,2)^{t}\rrbracket$.
By Rule~1, we get that 
  $\rmv \in \llbracket (2,3,2,1)^{t},(1,2,2,2)^{t}\rrbracket$.
Given that $1 = \varphi^{0}$, $2 \geq \varphi^{1}$ and $3 \geq \varphi^{2}$ 
  the claimed result follows.

We leave the last case to the interested reader.
\end{proof}

\begin{figure}[h]
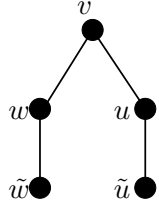

\centering
\ifpdf\input{or5case.pdf_t}\else\input{or5case.pstex_t}\fi
\caption{The tree $T_v$ of Proposition~\ref{prop:or5}.}\label{fig:or5case}
\end{figure}

\begin{proposition}\label{prop:or5} 
Let $T_v$ be a colored rooted ternary tree with $|T_v|=5$ and where
  $v$ has two children which are not leaves.
Then, there are $e_0,e_1,e_2,e_3\in\NN$ such that 
  $\rmv \geq \rme =(\varphi^{e_s})_{s=0,\ldots,3}$ 
  and $\Psi(\rme) \geq 8$. 
\end{proposition}
\begin{proof}
Assume $T_v$ is as depicted in Figure~\ref{fig:or5case}.
Clearly, $\rmw,\rmu \in \llbracket (1,2,1,1)^{t}\rrbracket$.
By Rule~2 we get that 
  $\rmv\in \llbracket (4,3,1,3)^{t},
       (1,3,3,2)^{t},(2,4,2,2)^{t},(2,2,1,5)^{t},(1,3,1,3)^{t},
       (1,3,4,3)^{t}\rrbracket$.
The desired conclusion follows since 
  $1 = \varphi^{0}$, $2 \geq \varphi^{1}$, $3 \geq \varphi^{2}$, 
  $4 \geq \varphi^{2}$ and $5 \geq \varphi^{3}$.
\end{proof}

\begin{proposition}\label{3chior5}  
Let $T_v$ be a colored rooted ternary tree such that $v$ has
  three children $u$, $w$ and $z$. 
Suppose that $1\leq|T_u|\leq3$ and $1\leq |T_{w}|\leq3$. Then,
\begin{itemize}
 \item If $|T_{z}|=2$, there are $e_0,e_1,e_2,e_3\in\NN$ such that 
  $\rmv \geq \rme=(\varphi^{e_s})_{s=0,\ldots,3}$ 
  and $\Psi(\rme) \geq 8$.
 \item If $|T_{z}|=3$, there are $e_0,e_1,e_2,e_3\in\NN$ such that 
  $\rmv \geq \rme=(\varphi^{e_s})_{s=0,\ldots,3}$ 
  and $\Psi(\rme) \geq 10$.
\end{itemize}
\end{proposition}
\begin{proof} 
We first note that if $|T_x|\geq2$, then
  $\rmx \geq (1,1,1,1)^{t}$. 
This implies that it is enough 
to prove both statements for the case $|T_u|=1$ 
  and $|T_{w}|=1$. Observe that by Rule~0, we have that 
  $\rmu = \rmw=(1,1,1,1)^{t}$.

For the first statement, assume $|T_{z}|=2$. 
By Rule~1, we have that 
  $\rmz\in\llbracket (1,2,1,1)^{t}\rrbracket$.
Then, by Rule~3 we have that
  $\rmv\in\llbracket (3,3,2,2)^{t}, (2,3,3,2)^{t}\rrbracket$.
The result follows, since $2 \geq \varphi^{1}$ and~$3 \geq \varphi^{2}$.

Assume now that $|T_{z}|=3$.
{From} the proof of 
  Proposition \ref{prop:or3} we know that 
  $\rmz \in \llbracket (2,3,1,1)^{t},(1,2,2,1)^{t}\rrbracket$.
Then, by Rule~3 we have that
  $\rmv\in\llbracket (5,2,5,2)^{t}, (3,4,3,4)^{t}, 
  (3,3,3,3)^{t}, (2,4,2,4)^{t}\rrbracket$.
The desired conclusion follows since 
  $2 \geq \varphi^{1}$, $3 \geq \varphi^{2}$, 
  $4 \geq \varphi^{2}$ and $5 \geq \varphi^{3}$.
\end{proof}

\begin{lemma}\label{lem:strip} 
Let $T = T_{v}$ be a colored rooted ternary tree, such that 
  $T_{v} = T_{\tilde{v}} \cup P_{\tilde{v},v}$ where $P_{\tilde{v},v}$ is non-trivial.
If $\tilde{\rmv}\geq
      \tilde{\rme}=(\varphi^{\tilde{e}_s})_{s=0,\ldots,3}$
  with $\tilde{e}_0,\tilde{e}_1,\tilde{e}_2,\tilde{e}_3\in\NN$, then
  there are $e_0,e_1,e_2,e_3 \in\NN$ such that
  $\rmv \geq \rme=(\varphi^{{e}_s})_{s=0,\ldots,3}$ and
  $\Phi(\rme) \geq \Phi(\tilde{\rme})+||P_{\tilde{v},v}||$.
\end{lemma}
\begin{proof} 
It is enough to prove the result for $P_{\tilde{v},v}$ of length 1.
By Rule~1, we get that 
  $\rmv\in [\tilde{\rmv}]$ 
  where $\tilde{\rmv} \geq \tilde{\rme}$ 
  for some $\tilde{\rme}\in \llbracket 
    (\varphi^{\tilde{e}_1}, \varphi^{\tilde{e}_1}+\varphi^{\tilde{e}_0},
     \varphi^{\tilde{e}_3}, \varphi^{\tilde{e}_2})^{t}\rrbracket$.
Assume $\tilde{\rme}$ is the vector within the double brackets 
  (the other cases are similar).
We now consider several scenarios:
\begin{itemize}
\item \textbf{Case $\tilde{e}_1>\tilde{e}_0+1$:} 
Clearly, $\rmv\geq \rme= (\varphi^{\tilde{e}_1}, \varphi^{\tilde{e}_1},
    \varphi^{\tilde{e}_3}, \varphi^{\tilde{e}_2})^{t}$.
Moreover, $\Psi(\rme) = \Psi(\tilde{\rme})$ and 
  $\left|\set{s}{e_{s}> e_{0}}\right|
    \leq \left|\set{s}{\tilde{e}_{s}> \tilde{e}_{0}}\right|-1$.
Hence, $\Phi(\rme) \geq \Phi(\tilde{\rme}) + 1$.

\item \textbf{Case $\tilde{e}_1\in\{\tilde{e}_0,\tilde{e}_0+1\}$:} 
If $\tilde{e}_1=\tilde{e}_0$, then 
  $\varphi^{\tilde{e}_1}+\varphi^{\tilde{e}_0}=2\varphi^{\tilde{e}_1}\geq 
     \varphi^{\tilde{e}_1+1}$.
Since $1+\varphi=\varphi^{2}$, if $\tilde{e}_1=\tilde{e}_0+1$, then
  $\varphi^{\tilde{e}_1}+\varphi^{\tilde{e}_0}=\varphi^{\tilde{e}_1+1}$.
Hence, $\rmv\geq\rme= (\varphi^{\tilde{e}_1},\varphi^{\tilde{e}_1+1},
  \varphi^{\tilde{e}_3},\varphi^{\tilde{e}_2})^{t}$. 
Moreover, $\Psi(\rme) = \Psi(\tilde{\rme})+2$ and 
  $\left|\set{s}{e_{s}> e_{0}}\right|\leq 
   \left|\set{s}{\tilde{e}_{s}> \tilde{e}_{0}}\right|+1$.
Hence, $\Phi(\rme) \geq \Phi(\tilde{\rme}) + 1$.

\item \textbf{Case $\tilde{e}_1 \leq \tilde{e}_0-1$:} 
Since $1+\varphi=\varphi^2$, if $\tilde{e}_1=\tilde{e}_0-1$, then
  $\varphi^{\tilde{e}_1}+\varphi^{\tilde{e}_0}=\varphi^{\tilde{e}_1+2}$.
If $\tilde{e}_1\leq\tilde{e}_0-2$, then
  $\varphi^{\tilde{e}_1}+\varphi^{\tilde{e}_0}\geq\varphi^{\tilde{e}_1+2}$.
Hence, $\rmv\geq\rme= (\varphi^{\tilde{e}_1},\varphi^{\tilde{e}_1+2},
   \varphi^{\tilde{e}_3},\varphi^{\tilde{e}_2})^{t}$.
Moreover, $\Psi(\rme) = \Psi(\tilde{\rme})+4$ and 
  $\left|\set{s}{e_{s}> e_{0}}\right|
    \leq \left|\set{s}{\tilde{e}_{s}> \tilde{e}_{0}}\right|+3$.
Hence, $\Phi(\rme) \geq \Phi(\tilde{\rme}) + 1$.
\end{itemize}
\end{proof}

The following result is an immediate consequence of Lemma~\ref{lem:strip}.
\begin{corollary}\label{onechild} 
Let $T = T_{v}$ be a colored rooted ternary tree, such that 
  $T_{v} = T_{\tilde{v}} \cup P_{\tilde{v},v}$ where $P_{\tilde{v},v}$ is non-trivial. 
If $\tilde{\rmv}\geq\tilde{\rme}$, then an $\rme$ exists 
  such that $\rmv \geq \rme$ and
\[
\Psi(\rme) \geq \Psi(\tilde{\rme})+
  \max\{||P_{\tilde{v},v}||-3,0\}\,.
\]
\end{corollary}

\begin{lemma}\label{twochildren}
Let $T_{v}$ be a colored rooted ternary tree, such that $v$ has 
  two children $w$ and $u$. 
If $\rmw\geq\rme^{w}=(\varphi^{e^{w}_s})_{s=0,\ldots,3}$, 
  with $e^{w}_0, e^{w}_1, e^{w}_2, e^{w}_3\in\NN$ and  
  $\rmu\geq\rme^{u}=(\varphi^{e^{u}_s})_{s=0,\ldots,3}$, 
  with $e^{u}_0, e^{u}_1, e^{u}_2,e^{u}_3\in\NN$, 
  then there are $e_0,e_1,e_2,e_3 \in\NN$ such that  
  $\rmv \geq \rme=(\varphi^{{e}_s})_{s=0,\ldots,3}$ and 
  $\Psi(\rme)=\Psi(\rme^{w})+\Psi(\rme^{u})$.
\end{lemma}
\begin{proof}
Since $\rmw\geq\rme^{w}$ and
  $\rmu\geq\rme^{u}$, by Rule~2 we have 
  that $\rmv\geq\tilde{\rmv}$ where 
\[
\tilde{\rmv}
  \in \left\{ \left( \begin{array}{c}
  \varphi^{e^{w}_1+e^{u}_1} \\
  \varphi^{e^{w}_0+e^{u}_2}  + \varphi^{e^{w}_1+e^{u}_3} \\
  \varphi^{e^{w}_3+e^{u}_2} \\
  \varphi^{e^{w}_2+e^{u}_1}  + \varphi^{e^{w}_3+e^{u}_0} 
  \end{array}\right), \left(\begin{array}{c}
  \varphi^{e^{w}_1+e^{u}_1} \\
  \varphi^{e^{w}_3+e^{u}_2} \\
  \varphi^{e^{w}_2+e^{u}_1}  + \varphi^{e^{w}_3+e^{u}_0}\\
  \varphi^{e^{w}_1+e^{u}_3} + \varphi^{e^{w}_0+e^{u}_2} 
  \end{array}\right), \left(\begin{array}{c}
  \varphi^{e^{w}_1+e^{u}_1} \\
  \varphi^{e^{w}_3+e^{u}_0}  + \varphi^{e^{w}_2+e^{u}_1}\\
  \varphi^{e^{w}_0+e^{u}_2} + \varphi^{e^{w}_1+e^{u}_3}\\
  \varphi^{e^{w}_3+e^{u}_2}
  \end{array}\right)  \right\}\,.
\]
Moreover, $\varphi^{e^{w}_0+e^{u}_2}+\varphi^{e^{w}_1+e^{u}_3} 
         \geq \varphi^{e^{w}_1+e^{u}_3}$
  and $\varphi^{e^{w}_2+e^{u}_1}+\varphi^{e^{w}_3+e^{u}_0}
         \geq \varphi^{e^{w}_2+e^{u}_1}$,
  so depending on the value of~$\tilde{\rmv}$ we can take 
\[
\rme \in \left\{
  \begin{pmatrix}
  \varphi^{e^{w}_1+e^{u}_1} \\
  \varphi^{e^{w}_1+e^{u}_3} \\
  \varphi^{e^{w}_3+e^{u}_2} \\
  \varphi^{e^{w}_2+e^{u}_1}   
  \end{pmatrix},
  \begin{pmatrix}
  \varphi^{e^{w}_1+e^{u}_1} \\
  \varphi^{e^{w}_3+e^{u}_2} \\
  \varphi^{e^{w}_2+e^{u}_1} \\
  \varphi^{e^{w}_1+e^{u}_3} 
  \end{pmatrix},
  \begin{pmatrix}
  \varphi^{e^{w}_1+e^{u}_1} \\
  \varphi^{e^{w}_2+e^{u}_1}\\
  \varphi^{e^{w}_1+e^{u}_3}\\
  \varphi^{e^{w}_3+e^{u}_2}
  \end{pmatrix}\right\}\,,
\]
and obtain that $\rmv\geq \rme$ and
  $\Psi(\rme)=\Psi(\rme^{w})+\Psi(\rme^{u})$.
\end{proof}

\begin{lemma}\label{threechildren}
Let $T_{v}$ be a colored rooted ternary tree, such that $v$ has three 
  children $w$, $u$ and $z$. 
If $\rmw\geq\rme^{w}=(\varphi^{e^{w}_s})_{s=0,\ldots,3}$ 
  with $e^{w}_0, e^{w}_1, e^{w}_2, e^{w}_3\in\NN$,  
  $\rmu\geq\rme^{u}=(\varphi^{e^{u}_s})_{s=0,\ldots,3}$ 
  with $e^{u}_0, e^{u}_1,e^{u}_2,e^{u}_3\in\NN$ and 
  $\rmz\geq\rme^{z}=(\varphi^{e^{z}_s})_{s=0,\ldots,3}$ 
  with $e^{z}_0, e^{z}_1,e^{z}_2,e^{z}_3\in\NN$,
  then there are  $e_0,e_1,e_2,e_3 \in\NN$ such that  
  $\rmv \geq \rme=(\varphi^{{e}_s})_{s=0,\ldots,3}$ 
  and $\Psi(\rme)=\Psi(\rme^{w})+\Psi(\rme^{u})+
  \Psi(\rme^{z})$.
\end{lemma}
\begin{proof}
By Rule~3 we have that 
\[
\rmv = \begin{pmatrix}
  \varphi^{e^{w}_0+e^{u}_0+e^{z}_0} + \varphi^{e^{w}_1+e^{u}_1+e^{z}_1}\\
  \varphi^{e^{w}_0+e^{u}_3+e^{z}_2} + \varphi^{e^{w}_1+e^{u}_3+e^{z}_2}\\
  \varphi^{e^{w}_2+e^{u}_3+e^{z}_0} + \varphi^{e^{w}_3+e^{u}_2+e^{z}_1}\\
  \varphi^{e^{w}_2+e^{u}_1+e^{z}_3} + \varphi^{e^{w}_3+e^{u}_0+e^{z}_2}
  \end{pmatrix} \geq 
  \begin{pmatrix}
  \varphi^{e^{w}_1+e^{u}_1+e^{z}_1}\\
  \varphi^{e^{w}_1+e^{u}_3+e^{z}_2}\\
  \varphi^{e^{w}_3+e^{u}_2+e^{z}_1}\\
  \varphi^{e^{w}_2+e^{u}_1+e^{z}_3} 
  \end{pmatrix}\,.
\]
Let $\rme$ be the last vector in the preceding expression 
  and note that 
  $\Psi(\rme)=\Psi(\rme^w)+\Psi(\rme^u)+\Psi(\rme^z)$.
\end{proof}

\subsection{Main Lemma} \label{subsec:mainlemma}
The main result of this work, i.e.~Theorem~\ref{theo:main}, will
  follow almost directly from the next key claim which roughly
  says that the root vector $\rmv$ of a colored rooted ternary remainder
  free tree $T=T_{v}$ either has large components relative to the size 
  of $T_{v}$, or $T_{v}$ corresponds to a \emph{short} path
  $P_{\tilde{v},v}$ and a tree $T_{\tilde{v}}$ whose root 
  vertex $\tilde{\rmv}$ has large components relative to the 
  size of $T_{\tilde{\rmv}}$.
\begin{lemma}\label{exp:sat} 
Let  $T = T_v$ be a colored rooted ternary remainder free tree such 
  that $|T|\geq4$. 
Then, there is a path $P_{\tilde{v},v}$ such that  
  $T_{v} = T_{\tilde{v}}\cup P_{\tilde{v},v}$ with 
  $0\leq ||P_{\tilde{v},v}||\leq 5$ (if $||P_{\tilde{v},v}||=0$, then
  $\tilde{v}=v$ and $T_{\tilde{v}}=T_{v}$)
  and  there are $e^{\tilde{v}}_0, e^{\tilde{v}}_1, e^{\tilde{v}}_2, 
  e^{\tilde{v}}_3 \in \NN$ such that
\begin{equation}\label{eqn:main}
\tilde{\rmv} \,\geq\, 
  \rme^{\tilde{v}} =(\varphi^{e^{\tilde{v}}_s})_{s=0,\ldots,3}\,, 
\quad\mbox{ and }\quad
\Psi(\rme^{\tilde{v}})  \,\geq\, \frac{|T_{\tilde{v}}|+7}{2}\,.
\end{equation}
\end{lemma}
\begin{proof} 
We proceed by induction on $|T|$. 
For the base case $|T_{v}|=4$, by  Proposition~\ref{prop:or4},
  there exists an $\rme\leq\rmv$ such that 
  $\Psi(\rme)\geq6 > (|T_{v}|+7)/2$.
Let $T_v$ be a colored rooted ternary tree remainder free with 
  $|T_v|\geq 5$. 
We separate the proof in three cases depending on the number 
  of children of the root~$v$. 
It is clear that for any node $u$ of $T_{v}$, the tree
  $T_u$ is a colored rooted ternary remainder free tree. 
\begin{description}
\item 
\textbf{Case 1 ($v$ has one child $w$):}
We have $|T_{w}|=|T_{v}|-1\geq 4$.
By induction
  $T_w = T_{\tilde{w}} \cup P_{\tilde{w},w}$ with
  $0\leq ||P_{\tilde{w},w}||\leq 5$ and 
  $\tilde{\rmw}$ satisfying~(\ref{eqn:main}).
If $||P_{w,\tilde{w}}||<5$,
  then $T_{v}=T_{\tilde{w}}\cup P_{\tilde{w},v}$ and thus it
  satisfies the desired property.
By Corollary~\ref{onechild}, we know that there is 
  and $\rme\leq\rmv$ such that $\Psi(\rme)
  \geq \Psi(\rme^{\tilde{w}})+||P_{\tilde{w},v}||-3$.
Given that $|T_{v}|=|T_{\tilde{w}}|+||P_{w,\tilde{w}}||+1$,
  if $||P_{w,\tilde{w}}||=5$, then
  there is an $\rme \leq \rmv$ such that
\[
\Psi(\rme) \geq \Psi(\rme^{\tilde{w}}) + ||P_{\tilde{w},v}||-3
  \geq \frac{|T_{\tilde{w}}|+7}{2} + ||P_{\tilde{w},v}||-3
  = \frac{|T_{v}|+7}{2}\,.
\]
Therefore, $T_v$ satisfies the desired property.

\item \textbf{Case 2 ($v$ has two children $w$ and $u$):} 
First, note that $T_w$ and $T_u$ have size at least $2$ (otherwise 
  we would have, say $|T_{w}|=1$ and $|T_{u}|=|T_{v}|-|T_{w}|-1\geq 3$, implying
  that $w$ is a remainder of $T$, a contradiction).
If $|T_{w}|=2$, then $|T_{u}|=2$ (otherwise, $|T_{u}|\geq 3$, implying that  
  there is a remainder~$S$ of~$T$ such that 
  $w\in S$, a contradiction).
Since $|T_{w}|=|T_{u}|=2$, by Proposition~\ref{prop:or5}
  we have that there is an $\rme\leq\rmv$ such that
  $\Psi(\rme)=8> (|T_{v}|+7)/2$. 

Hence, we assume that $|T_{w}|, |T_{u}| \geq 3$. 
If $|T_{w}|= |T_{u}| = 3$, by Proposition~\ref{prop:or3} and 
  Lemma~\ref{twochildren}, we get that there is an
  $\rme\leq\rmv$ such that 
  $\Psi(\rme)=8> (|T_{v}|+7)/2$. 

We now assume that $|T_{w}| = 3$ and $|T_{u}| \geq 4$. 
By induction, $T_{u} = T_{\tilde{u}}\cup P_{\tilde{u},u}$ 
  with $0\leq||P_{\tilde{u},u}||\leq 5$ 
  and~$\tilde{\rmu}$ satisfying~(\ref{eqn:main}).
By Lemma~\ref{twochildren}, there is an
  $\rme\leq\rmv$ such that 
  $\Psi(\rme) = \Psi(\rme^{w}) + \Psi(\rme^{u})$.
By Proposition~\ref{prop:or3}, Corollary~\ref{onechild}, and the fact that
  $|T_{v}|=|T_{\tilde{u}}|+||P_{\tilde{u},u}||+4$, 
\begin{eqnarray*}
\Psi(\rme) & \geq & \Psi(\rme^{w})+\Psi(\rme^{\tilde{u}})
    + \max\{||P_{\tilde{u},u}||-3,0\} \\
  & \geq & 4 + \frac{|T_{\tilde{u}}|+7}{2}
    + \max\{||P_{\tilde{u},u}||-3,0\} \\
  & = & \frac{|T_{v}|+7}{2}+\frac{4-||P_{\tilde{u},u}||}{2}
    + \max\{||P_{\tilde{u},u}||-3,0\}  \\
  & \geq & \frac{|T_{v}|+7}{2}
    +\frac{1}{2}\max\{3-||P_{\tilde{u},u}||,||P_{\tilde{u},u}||-3,0\} \\
  & \geq & \frac{|T_{v}|+7}{2}\,.
\end{eqnarray*}
Hence, $T_{v}$ satisfies the  desired property.

Finally, we assume that $|T_{w}|, |T_{u}| \geq 4$. 
By induction, $T_{w}= T_{\tilde{w}} \cup P_{\tilde{w},w}$ and 
  $T_{u}= T_{\tilde{u}} \cup P_{\tilde{u},u}$
  where  $0\leq ||P_{\tilde{w},w}||,||P_{\tilde{u},u}||\leq 5$
  and $\tilde{u},\tilde{w}$ satisfying~(\ref{eqn:main}).
By Lemma~\ref{twochildren}, there is an
  $\rme\leq\rmv$ such that 
  $\Psi(\rme) = \Psi(\rme^{w}) + \Psi(\rme^{u})$.
By Corollary~\ref{onechild}
  and given that $|T_{v}|=|T_{\tilde{w}}|+|T_{\tilde{u}}|
  +||P_{\tilde{w},w}||+||P_{\tilde{u},u}||+1$,
\begin{eqnarray*}
\Psi(\rme) 
  & \geq & \Psi(\rme^{\tilde{w}}) + \Psi(\rme^{\tilde{u}})+
       +\max\{||P_{\tilde{w},w}||{-}3,0\}+\max\{||P_{\tilde{u},u}||{-}3,0\} \\
  & \geq & \frac{|T_{\tilde{w}}|+7}{2}+\frac{|T_{\tilde{u}}|+7}{2}
       +\max\{||P_{\tilde{w},w}||{-}3,0\}+\max\{||P_{\tilde{u},u}||{-}3,0\} \\
  & = &  \frac{|T_{v}|+7}{2}
    +\frac{1}{2}\max\{||P_{\tilde{w},w}||{-}3,3{-}||P_{\tilde{w},w}||\}
    +\frac{1}{2}\max\{||P_{\tilde{u},u}||{-}3,3{-}||P_{\tilde{u},u}||\} \\
  & \geq & \frac{T_{v}+7}{2}\,.
\end{eqnarray*}
Hence, $T_{v}$ satisfies the desired property.

\item \textbf{Case 3 ($v$ has three children $w$, $u$ and $z$):} 
Since $|T_{v}|\geq 5$, it can not happen that 
  $|T_{u}|=|T_{u}|=|T_{z}|=1$.
If $1\leq|T_w|\leq3$, $1\leq|T_u|\leq3$ and 
  $2\leq|T_z|\leq3$, we have that: if $|T_z|=2$, then $|T_{v}| \leq 9$ and
  by the first statement of Proposition~\ref{3chior5} there is a vector 
  $\rme\leq\rmv$ such that 
  $\Psi(\rme) \geq 8 = 16/2 \geq (|T_{v}|+7)/2$; if~$|T_z|=3$, 
  then $|T_{v}| \leq 10$ and
  by the second statement of Proposition~\ref{3chior5} there is a 
  vector $\rme\leq\rmv$ such that 
  $\Psi(\rme) \geq 10 > 17/2 \geq (|T_{v}|+7)/2$.
Therefore, $T_v$ satisfies the desired property.

We now assume that at least one of the children of $v$ induces a subtree 
  with at least 4 vertices. 
\begin{itemize}
\item If $1\leq |T_{w}|,|T_u|\leq 2$ and $|T_{z}|\geq4$,
  then by Rules~0,~1 and~3, we have
\[
\rmv \geq 
  \begin{pmatrix}
  \varphi^{e^{z}_0} +\varphi^{e^{z}_1}\\
  \varphi^{e^{z}_3} + \varphi^{e^{z}_2}\\
  \varphi^{e^{z}_0} +  \varphi^{e^{z}_1}\\
  \varphi^{e^{z}_3} +  \varphi^{e^{z}_2}
  \end{pmatrix}\,,
\quad\text{or}\quad
\rmv \geq 
  \begin{pmatrix}
  \varphi^{e^{z}_0} +  \varphi^{e^{z}_1}\\
  \varphi^{e^{z}_2} +  \varphi^{e^{z}_3}\\
  \varphi^{e^{z}_3} +  \varphi^{e^{z}_2}\\
  \varphi^{e^{z}_1} +  \varphi^{e^{z}_0}
  \end{pmatrix}\,,
\quad\text{or}\quad
\rmv \geq
   \begin{pmatrix}
   \varphi^{e^{z}_0} +  \varphi^{e^{z}_1}\\
   \varphi^{e^{z}_0} +  \varphi^{e^{z}_1}\\
   \varphi^{e^{z}_2} +  \varphi^{e^{z}_3}\\
   \varphi^{e^{z}_2} +  \varphi^{e^{z}_3}
  \end{pmatrix}\,.
\]
If  $e^{z}_3=e^{z}_2$, given that $2>\varphi$, we may choose the vector 
  $\rme$ from the set
\[
\left\{  
  \begin{pmatrix}
  \varphi^{e^{z}_1} \\
  \varphi^{e^{z}_3+1} \\
  \varphi^{e^{z}_1} \\
  \varphi^{e^{z}_2+1}  
  \end{pmatrix}, 
  \begin{pmatrix}
  \varphi^{e^{z}_1} \\
  \varphi^{e^{z}_3+1} \\
  \varphi^{e^{z}_2+1}\\
  \varphi^{e^{z}_1} 
  \end{pmatrix}, 
  \begin{pmatrix}
  \varphi^{e^{z}_1} \\
  \varphi^{e^{z}_1}\\
  \varphi^{e^{z}_2+1}\\
  \varphi^{e^{z}_3+1}
  \end{pmatrix}
\right\}\,.
\]
If not, we have $e^{z}_3\geq e^{z}_2+1$ 
  (analogously $e^{z}_2\geq e^{z}_3+1$) and given that 
  $\varphi+1=\varphi^2$, we may choose the vector $\rme$ 
  from the set
\[
\left\{  
  \begin{pmatrix}
  \varphi^{e^{z}_1} \\
  \varphi^{e^{z}_2+2} \\
  \varphi^{e^{z}_1} \\
  \varphi^{e^{z}_3}  
  \end{pmatrix}, 
  \begin{pmatrix}
  \varphi^{e^{z}_1} \\
  \varphi^{e^{z}_3} \\
  \varphi^{e^{z}_2+2}\\
  \varphi^{e^{z}_1} 
  \end{pmatrix},
  \begin{pmatrix}
  \varphi^{e^{z}_1} \\
  \varphi^{e^{z}_1}\\
  \varphi^{e^{z}_3}\\
  \varphi^{e^{z}_2+2}
  \end{pmatrix}
\right\}.
\]
Therefore, for any choice of $\rme$ we get that
  $\Psi(\rme)= \Psi(\rme^{z}) + 4$. 
By induction, $T_{z} = T_{\tilde{z}}\cup P_{\tilde{z},z}$ 
  with $0\leq||P_{\tilde{z},z}||\leq 5$ and 
  $\tilde{\rmz}$ satisfying~(\ref{eqn:main}).
Since  $|T_{v}|\leq |T_{\tilde{z}}|+||P_{\tilde{z},z}||+5$,
  by Corollary~\ref{onechild}, 
\begin{eqnarray*}
\Psi(\rme)
  & \geq & 4+\Psi(\rme^{\tilde{z}}) +
           \max\{||P_{\tilde{z},z}||-3,0\} \\
  & \geq & \frac{|T_{\tilde{z}}|+7}{2}+4+
           \max\{||P_{\tilde{z},z}||-3,0\} \\
  & \geq & \frac{|T_{v}|+7}{2}
        +\frac{3-||P_{\tilde{z},z}||}{2}
        +\max\{||P_{\tilde{z},z}||-3,0\} \\
  & = & \frac{|T_{v}|+7}{2}
    +\frac{1}{2}\max\{3-||P_{\tilde{z},z}||,||P_{\tilde{z},z}||-3\} \\
  & \geq & \frac{|T_{v}|+7}{2}\,.
\end{eqnarray*}
Hence, $T_{v}$ satisfies the desired property.

\item If $2\leq |T_w|,|T_{u}|\leq 3$ and $|T_{z}|\geq 4$. 
By Lemma~\ref{threechildren}, there is an
  $\rme\leq\rmv$ such that 
  $\Psi(\rme)=\Psi(\rme^{w})+\Psi(\rme^{u})+
    \Psi(\rme^{z})$.
By Proposition~\ref{prop:or2} and Proposition~\ref{prop:or3}, we have 
  $\Psi(\rme^{w})=2(|T_{w}|-1)$ and
  $\Psi(\rme^{u})=2(|T_{u}|-1)$.
By induction, $T_{z} = T_{\tilde{z}}\cup P_{\tilde{z},z}$ 
  with $0\leq||P_{\tilde{z},z}||\leq 5$ and $\tilde{\rmz}$ 
  satisfying~(\ref{eqn:main}).
Since $|T_{v}|= |T_{w}|+|T_{u}|+|T_{\tilde{z}}|+||P_{\tilde{z},z}||+1$, 
  by Corollary~\ref{onechild}, 
\begin{eqnarray*}
\Psi(\rme)
  & \geq & 2(|T_{w}|-1)+2(|T_{u}|-1)+\Psi(\rme^{\tilde{z}}) 
     +\max\{||P_{\tilde{z},z}||-3,0\} \\
  & \geq & \frac{|T_{\tilde{z}}|+7}{2}+2(|T_{w}|+|T_{u}|)-4
     +\max\{||P_{\tilde{z},z}||-3,0\} \\
  & = & \frac{|T_{v}|+7}{2}+
    \frac{3}{2}(|T_{w}|+|T_{u}|)-\frac{||P_{\tilde{z},z}||}{2}-\frac{9}{2}
    +\max\{||P_{\tilde{z},z}||-3,0\} \\
  & \geq & \frac{|T_{v}|+7}{2}+
    \frac{3-||P_{\tilde{z},z}||}{2}
    +\max\{||P_{\tilde{z},z}||-3,0\} \\
  & = & \frac{|T_{v}|+7}{2}
    +\frac{1}{2}\max\{3-||P_{\tilde{z},z}||,||P_{\tilde{z},z}||-3\} \\
  & \geq & \frac{|T_{v}|+7}{2}\,.
\end{eqnarray*}
Hence, $T_{v}$ satisfies the desired property.

\item The case $|T_w|=1$, $|T_u|=3$, and $|T_{z}|\geq4$ can not 
  happen, since it would imply that $\{w\}$ is a remainder of 
  $T$.

\item If $2\leq |T_w|\leq 3$ and $|T_u|, |T_{z}|\geq 4$. 
By Lemma~\ref{threechildren}, there is an
  $\rme\leq\rmv$ such that 
  $\Psi(\rme)=\Psi(\rme^{w})+\Psi(\rme^{u})+
    \Psi(\rme^{z})$.
By Proposition~\ref{prop:or2} and Proposition~\ref{prop:or3}, we have 
  $\Psi(\rme^{w})=2(|T_{w}|-1)$.
By induction, $T_{z} = T_{\tilde{z}}\cup P_{\tilde{z},z}$ and 
  $T_{z} = T_{\tilde{z}}\cup P_{\tilde{z},z}$ 
  with $0\leq ||P_{\tilde{u},u}||,||P_{\tilde{z},z}||\leq 5$ and 
  $\tilde{\rmu},\tilde{\rmz}$ satisfying~(\ref{eqn:main}).
Since $|T_{v}|= |T_{w}|+|T_{\tilde{u}}|+|T_{\tilde{z}}|+||P_{\tilde{u},u}||
  +||P_{\tilde{z},z}||+1$, by Corollary~\ref{onechild}, 
\begin{eqnarray*}
\Psi(\rme)
  & \geq & 2(|T_{w}|-1)+\Psi(\rme^{\tilde{u}})+\Psi(\rme^{\tilde{z}})  
     +\max\{||P_{\tilde{u},u}||-3,0\} 
     +\max\{||P_{\tilde{z},z}||-3,0\} \\
  & \geq & \frac{|T_{\tilde{u}}|+7}{2}+\frac{|T_{\tilde{z}}|+7}{2}+
     2(|T_{w}|-1) \\
  & & \mbox{}\quad
     +\max\{||P_{\tilde{u},u}||-3,0\} 
     +\max\{||P_{\tilde{z},z}||-3,0\} \\
  & = & \frac{|T_{v}|+7}{2}+
    \frac{3}{2}|T_{w}|-2+\frac{6-||P_{\tilde{u},u}||-||P_{\tilde{z},z}||}{2} \\
  & & \mbox{}\quad
     +\max\{||P_{\tilde{u},u}||-3,0\} 
     +\max\{||P_{\tilde{z},z}||-3,0\} \\
  & \geq & \frac{|T_{v}|+7}{2}+
     \frac{6-||P_{\tilde{u},u}||-||P_{\tilde{z},z}||}{2} \\
  & & \mbox{}\quad
     +\max\{||P_{\tilde{u},u}||-3,0\} 
     +\max\{||P_{\tilde{z},z}||-3,0\} \\
  & = & \frac{|T_{v}|+7}{2}
    +\frac{1}{2}\max\{3-||P_{\tilde{z},z}||,||P_{\tilde{z},z}||-3\} \\
  & \geq & \frac{|T_{v}|+7}{2}\,.
\end{eqnarray*}
Hence, $T_{v}$ satisfies the desired property.

\item If $|T_w|, |T_u|, |T_{z}|\geq 4$. 
Similar to the preceding case.
\end{itemize}
\end{description}
\end{proof}

\section{Proof of Main Results}\label{sec:results}
\begin{prooff}{Theorem~\ref{theo:main}}
Recall that $T(\Delta_n)$ is a colored rooted ternary tree on
  $|\Delta_{n}|-3$ nodes such that its root vector $\rmv$ 
  is equal to the degeneracy vector of $\Delta_n$.
By Lemmas~\ref{inters} and~\ref{lemma:b}, the rooted colored ternary 
  tree $\tilde{T}(\Delta_n)=T(\Delta_n)\setminus V_{R(T(\Delta_n))}$ is 
  remainder free and $|\tilde{T}(\Delta_n)|\geq |T(\Delta_n)|/3$.
Clearly, the root vector~$\tilde{\rmv}$ of~$\tilde{T}(\Delta_n)$
  is such that $\rmv\geq \tilde{\rmv}$.
The Main Lemma guarantees that there are $e_0,e_1,e_2,e_3 \in\NN$ such that
  $\rmv\geq \rme=(\varphi^{e_s})_{s=0,..,3}$ and 
\[
\Psi(\rme) \,\geq\, \frac{(|\tilde{T}(\Delta_n)|-5)+7}{2}
  \,=\, \frac{|\tilde{T}(\Delta_n)|+2}{2}
  \,\geq\, \frac{|T(\Delta_n)|+6}{12}
  \,=\, \frac{|\Delta_n|+3}{12}\,.
\]
Moreover, we know that $\Delta_{n}[\phi] = \Delta_{n}[-\phi]$ 
  for all $\phi \in \{+,-\}^3$.
Hence, the degeneracy of $\Delta_{n}$ is at least 
  $2\sum_{s=1}^{3} \varphi^{e_s} \geq 6\varphi^{\frac{1}{3}\Psi(\rme)} \geq
  6 \varphi^{(|\Delta_n|+3)/36}$.
\end{prooff}

\begin{prooff}{Corollary~\ref{coro:main}}  
Let $G$ be a cubic planar graph such that its geometric dual 
  graph is the stack triangulation $\Delta$. 
We know that the number of perfect matchings of $G$ is equal to half 
  of the degeneracy of $\Delta$. 
From Euler's formula we get that $2|\Delta|=|G|-4$. 
Therefore, by Theorem~\ref{theo:main} we have that the number of 
  perfect matchings of $G$ is at least  $3 \varphi^{|G|/72}$. 
\end{prooff}

\section{Final Comments}\label{sec:final-comments}
The approach followed throughout this work seems to be specially
  well suited for calculating the degeneracy of triangulations 
  that have some sort of recursive tree like construction, 
  e.g.~$3$-trees.
It would be interesting to identify other such families of 
  triangulations where similar methods allowed to lower bound
  their degeneracy. 
Of particular relevance would be to show that the approach we follow
  in this work can actually be successfully applied to obtain exponential lower
  bounds for non-trivial families of non-planar bridgeless cubic graphs.

As already mentioned, our arguments are motivated by the transfer 
  matrix method as used by statistical physicists. 
We believe that most of the arguments we developed throughout
  this work can be stated in more combinatorial terms, except
  maybe for our Main Lemma.
It might eventually be worthwhile to clarify the implicit
  combinatorial structure of our proof arguments.

\section*{Acknowledgement}
The authors thank Martin Loebl for his encouragement, comments, and many
  insightful discussions.

\bibliographystyle{plain}
\bibliography{biblio}
\end{document}